\documentclass[11pt,leqno]{article}
\usepackage{amsmath, amscd, amsthm, amssymb, graphics, xypic, mathrsfs, setspace, fancyhdr, times, enumerate}
\entrymodifiers={+!!<0pt,\fontdimen22\textfont2>}
\DeclareFontFamily{OT1}{pzc}{}
\DeclareFontShape{OT1}{pzc}{m}{it}%
             {<-> s * [1.195] pzcmi7t}{}
\DeclareMathAlphabet{\mathscr}{OT1}{pzc}%
                                 {m}{it}
\usepackage[colorlinks=true,pagebackref=true]{hyperref} 
\hypersetup{backref}

\newcommand{\Spec}{\operatorname{Spec}}

\newcommand{\sma}{{\scriptstyle{\wedge}}}

\newcommand{\real}{{\mathbb R}}

\newcommand{\Z}{{\mathbb Z}}
\newcommand{\N}{{\mathbb N}}
\newcommand{\A}{{\mathbb A}}

\newcommand{\aone}{{\mathbb A}^1}
\newcommand{\pone}{{\mathbb P}^1}

\renewcommand{\L}{{\mathcal L}}
\newcommand{\MW}{\mathrm{MW}}
\newcommand{\Mi}{\mathrm{M}}

\newcommand{\et}{\text{\'et}}

\newcommand{\sI}{{\overline{\mathbf I}}}  

\newcommand{\CH}{{\widetilde{CH}}}

\newcommand{\K}{{{\mathbf K}}}

\renewcommand{\H}{{{\mathbf H}}}

\newcounter{intro}
\theoremstyle{plain}
\newtheorem{thm}{Theorem}[subsection]

\newtheorem{lem}[thm]{Lemma}
\newtheorem{cor}[thm]{Corollary}
\newtheorem{prop}[thm]{Proposition}
\newtheorem*{claim*}{Claim}  

\newtheorem*{thm*}{Theorem}
\newtheorem*{problem*}{Problem}

\newtheorem{thmintro}{Theorem}

\newtheorem{corintro}[thmintro]{Corollary}

\theoremstyle{definition}
\newtheorem{defn}[thm]{Definition}

\theoremstyle{remark}
\newtheorem{rem}[thm]{Remark}
\newtheorem{remintro}[thmintro]{Remark}

\numberwithin{equation}{subsection}

\begin{document}
\pagestyle{fancy}
\renewcommand{\sectionmark}[1]{\markright{\thesection\ #1}}
\fancyhead{}
\fancyhead[LO,R]{\bfseries\footnotesize\thepage}
\fancyhead[LE]{\bfseries\footnotesize\rightmark}
\fancyhead[RO]{\bfseries\footnotesize\rightmark}
\chead[]{}
\cfoot[]{}
\setlength{\headheight}{1cm}

\author{\begin{small}Aravind Asok\thanks{Aravind Asok was partially supported by National Science Foundation Awards DMS-0966589 and DMS-1254892.}\end{small} \\ \begin{footnotesize}Department of Mathematics\end{footnotesize} \\ \begin{footnotesize}University of Southern California\end{footnotesize} \\ \begin{footnotesize}Los Angeles, CA 90089-2532 \end{footnotesize} \\ \begin{footnotesize}\url{asok@usc.edu}\end{footnotesize} \and \begin{small}Jean Fasel\thanks{Jean Fasel was partially supported by the DFG Grant SFB Transregio 45.}\end{small} \\ \begin{footnotesize}Institut Fourier-UMR 5582\end{footnotesize} \\\begin{footnotesize} 100 rue des math\'ematiques\end{footnotesize}\\ \begin{footnotesize}38402 Saint-Martin d'H{\`e}res\end{footnotesize} \\ \begin{footnotesize}\url{jean.fasel@gmail.com}\end{footnotesize}}

\title{{\bf Secondary characteristic classes \\ and the Euler class}}
\date{}
\maketitle

\begin{abstract}
We discuss secondary (and higher) characteristic classes for algebraic vector bundles with trivial top Chern class.  We then show that if $X$ is a smooth affine scheme of dimension $d$ over a field $k$ of finite $2$-cohomological dimension (with $\mathrm{char}(k)\neq 2$) and $E$ is a rank $d$ vector bundle over $X$, vanishing of the Chow-Witt theoretic Euler class of $E$ is equivalent to vanishing of its top Chern class and these higher classes. We then derive some consequences of our main theorem when $k$ is of small $2$-cohomological dimension.
\end{abstract}

\begin{footnotesize}
\setcounter{tocdepth}{1}
\tableofcontents
\end{footnotesize}
\section{Introduction}
\label{s:introduction}
Suppose $k$ is a field having characteristic unequal to $2$, $X=\Spec(A)$ is a $d$-dimensional smooth affine $k$-scheme and $\mathcal E$ is a vector bundle of rank $r$ over $X$.  There is a well-defined primary obstruction to $\mathcal E$ splitting off a free rank $1$ summand given by ``the" Euler class $e(\mathcal E)$ of $\mathcal E$ (see \cite[Theorem 8.2]{MField}, \cite[Chapitre 13]{FaselChowWitt} and \cite{AsokFaselComparison}, which shows two possible definitions coincide for oriented vector bundles).  When $r = d$, Morel shows that this primary obstruction is the only obstruction to splitting off a trivial rank $1$ summand, and we will focus on this case in this article.

Because the Euler class is defined using Chow-Witt theory, which is not part of an oriented cohomology theory (say in the sense of \cite{LevineMorel}), it is difficult to compute in general.  The vanishing of the Euler class implies the vanishing of the top Chern class $c_d(\mathcal E)$ in $CH^d(X)$ \cite[Proposition 6.3.1]{AsokFaselA3minus0}, though the converse is not true in general. It is therefore natural to try to approximate $e(\mathcal E)$ using structures defined only in terms of oriented cohomology theories. More precisely, we now explain the strategy involved in studying such ``approximations" as developed in Section \ref{sec:Pardon}.

If $X$ is as above, let us fix a line bundle $\L$ on $X$.  One can define the $\L$-twisted unramified Milnor-Witt K-theory sheaf $\K^{\MW}_d(\L)$, which is a sheaf on the small Nisnevich site of $X$.  The $\L$-twisted Chow-Witt group $\CH^d(X,\mathcal L)$ can be defined as the Nisnevich cohomology group $H^d(X,\K^{\MW}_d(\mathcal L))$.  With $\mathcal{E}$ as above, the Euler class $e(\mathcal{E})$ lives in this group with $\L = \det \mathcal{E}^{\vee}$.

If $\K^{\Mi}_d$ is the $d$-th unramified Milnor K-theory sheaf, then by Rost's formula $H^d(X,\K^{\Mi}_d) \cong CH^d(X)$.  There is a natural morphism of sheaves on $X$ of the form $\K^{\MW}_d(\mathcal L)\to \K^{\Mi}_d$, which furnishes a comparison morphism $\widetilde{CH}^d(X,\L) \to CH^d(X)$ whose study is the main goal of this paper.

By a result of F. Morel, the kernel of the morphism of sheaves $\K^{\MW}_d(\mathcal L)\to \K^{\Mi}_d$ is the $(d+1)$st power of the fundamental ideal in the Witt sheaf (twisted by $\mathcal L$), denoted $\mathbf{I}^{d+1}(\mathcal L)$.  The sheaf $\mathbf{I}^{d+1}(\mathcal L)$ is filtered by subsheaves of the form $\mathbf{I}^{r}(\mathcal L)$ for $r \geq d+1$:
\[
\ldots\subset \mathbf{I}^{n+d}(\mathcal L)\subset \mathbf{I}^{n+d-1}(\mathcal L)\subset\ldots\subset \mathbf{I}^{d+1}(\mathcal L)\subset \K^{\MW}_d(\mathcal L).
\]
This filtration induces associated long exact sequences in cohomology and gives rise to a spectral sequence $E(\mathcal L,\MW)^{p,q}$ computing the cohomology groups with coefficients in $\K^{\MW}_d(\mathcal L)$.

When $p=d=\mathrm{dim}(X)$, we obtain a filtration of the group $H^d(X,\K^{\MW}_d(\mathcal L))$ by subgroups $F^nH^d(X,\K^{\MW}_d(\mathcal L))$ for $n\in\N$ such that $F^0H^d(X,\K^{\MW}_d(\mathcal L))=H^d(X,\K^{\MW}_d(\mathcal L))$ and where the successive subquotients $F^nH^d(X,\K^{\MW}_d(\mathcal L))/F^{n+1}H^d(X,\K^{\MW}_d(\mathcal L))$ are computed by the groups $E(\mathcal L,\MW)^{d,d+n}_\infty$ arising in the spectral sequence.  If furthermore $k$ has finite $2$-cohomological dimension, then only finitely many of the groups $E(\mathcal L,\MW)^{d,d+n}_\infty$ are nontrivial and we obtain the following theorem.

\begin{thmintro}[{See Theorem \ref{thm:obstructions}}]
Suppose $k$ is a field having finite $2$-cohomological dimension (and having characteristic unequal to $2$).  Suppose $X$ is a smooth $k$-scheme of dimension $d$ and suppose $\mathcal L$ is line bundle on $X$.  For any $\alpha\in H^d(X,\K_d^{\MW}(\mathcal L))$, there are inductively defined obstructions $\Psi^n(\alpha)\in E(\mathcal L,\MW)_{\infty}^{d,d+n}$ for $n\geq 0$ such that $\alpha=0$ if and only if $\Psi^n(\alpha)=0$ for any $n\geq 0$.
\end{thmintro}

The groups $E(\mathcal L,\MW)_2^{p,q}$ are cohomology groups with coefficients either in $\K_d^{\Mi}$ or in $\K_j^{\Mi}/2$ for $j\geq d+1$, and thus they are theoretically easier to compute than the cohomology groups with coefficients in $\K^{\MW}_d$; this is the sense in which we have ``approximated" our original non-oriented computation by ``oriented" computations.  The upshot is that if $k$ has finite $2$-cohomological dimension, we can use a vanishing result from \cite{AsokFaselThreefolds} (which appeals to Voevodsky's resolution of the Milnor conjecture on the mod $2$ norm-residue homomorphism) to establish the following result.

\begin{corintro}
\label{corintro:finitelymanyobstructions}
Let $k$ be a field having $2$-cohomological dimension $s$ (and having characteristic unequal to $2$). If $X$ is a smooth affine $k$-scheme of dimension $d$ and $\xi:\mathcal E\to X$ is a rank $d$-vector bundle on $X$ with $c_d(\mathcal E)=0$, then $\mathcal E$ splits off a trivial rank $1$ summand if and only if $\Psi^n(\mathcal E) = 0$ for $n\leq s-1$.
\end{corintro}

The problem that arises then is to identify the differentials in the spectral sequence, which provide the requisite ``higher obstructions", in concrete terms.  To this end, we first observe that there is a commutative diagram of filtrations by subsheaves
\[
\xymatrix{\ldots\ar[r] & \mathbf{I}^{d+n}(\mathcal L)\ar[r]\ar@{=}[d] & \mathbf{I}^{d+n-1}(\mathcal L)\ar[r]\ar@{=}[d] & \ldots \ar[r] & \mathbf{I}^{d+1}(\mathcal L)\ar@{=}[d]\ar[r] & \K^{\MW}_d(\mathcal L)\ar[d] \\
\ldots\ar[r] & \mathbf{I}^{d+n}(\mathcal L)\ar[r] & \mathbf{I}^{d+n-1}(\mathcal L)\ar[r] & \ldots \ar[r] & \mathbf{I}^{d+1}(\mathcal L)\ar[r] & \mathbf{I}^d(\mathcal L).}
\]
The filtration on the bottom gives rise to (a truncated version of) the spectral sequence Pardon studied \cite[0.13]{PardonGW}; this spectral sequence was further analyzed in \cite{Totaro}.  Totaro showed that the differentials on the main diagonal in the $E_2$-page of the Pardon spectral sequence are given by Voevodsky's Steenrod squaring operation $Sq^2$.  Using the diagram above, we see that the differentials in the spectral sequence we define are essentially determined by the differentials in the Pardon spectral sequence, and we focus on the latter.  We extend Totaro's results and obtain a description of the differentials just above the main diagonal as well and, more generally, the differentials in our $\L$-twisted spectral sequence (see Theorem \ref{thm:main}).

We identify, using the Milnor conjecture on the mod $2$ norm-residue homomorphism, the (mod $2$) Milnor K-cohomology groups appearing in the pages of the spectral sequence above in terms of motivic cohomology groups.  Via this identification, the differentials appearing just above the main diagonal in our spectral sequence can be viewed as operations on motivic cohomology groups.  Bi-stable operations of mod $2$ motivic cohomology groups have been identified by Voevodsky \cite{VMEMSpaces} (if $k$ has characteristic $0$) or Hoyois-Kelly-{\O}stvaer \cite{HKO} (if $k$ has characteristic unequal to $2$).  It follows from these identifications that the differentials in question are either the trivial operation or the (twisted) Steenrod square.  In Section \ref{ss:nontriviality}, we compute an explicit example to rule out the case that the operation is trivial.  Finally, we put everything together in the last section to obtain, in particular, the following result.

\begin{thmintro}
\label{thmintro:secondaryclass}
Let $k$ be a field having $2$-cohomological dimension $s$ (and having characteristic unequal to $2$). Suppose $X$ is a smooth affine $k$-scheme of dimension $d$ and $\xi:\mathcal E\to X$ is a rank $d$-vector bundle on $X$ with $c_d(\mathcal E)=0$.  The secondary obstruction $\Psi^1(\alpha)$ to $\mathcal{E}$ splitting off a trivial rank $1$ summand is the class in the cokernel of the composite map
\[
H^{d-1}(X,\K_d^{\Mi}) \longrightarrow H^{d-1}(X,\K_d^{\Mi}/2) \stackrel{Sq^2 + c_1(\L) \cup}{\longrightarrow } H^d(X,\K_{d+1}^{\Mi}/2),
\]
(the first map is induced by reduction mod $2$) defined as follows: choose a lift of the class $e(\xi) \in H^d(X,\mathbf{I}^{d+1}(\det \mathcal{E}))$ and look at its image in $H^d(X,\K^{\Mi}_{d+1}/2)$ under the map $\mathbf{I}^{d+1}(\det \mathcal{E})) \to \K^{\Mi}_{d+1}/2$.  Furthermore: (i) if $k$ has cohomological dimension $1$, then the secondary (and all higher) obstructions are automatically trivial and (ii) if $k$ has cohomological dimension $2$, then the triviality of the secondary obstruction is the only obstruction to $\mathcal{E}$ splitting off a trivial rank $1$ summand.
\end{thmintro}

For the sake of perspective, recall that Bhatwadekar and Sridharan asked whether the only obstruction to splitting a trivial rank $1$ summand off a rank $(2n+1)$ vector bundle $\mathcal{E}$ on a smooth affine $(2n+1)$-fold $X = \Spec A$ is vanishing of a variant of the top Chern class living in a group $E_0(A)$ \cite[Question 7.12]{BhatwadekarSridharan00}.  The group $E_0(A)$ housing their obstruction class is isomorphic to the Chow group of $0$-cycles on $\Spec A$ in some cases; see, e.g., \cite[Remark 3.13 and Theorem 5.5]{BhatwadekarSridharan99}.  It is an open problem whether the group $E_0(A)$ is isomorphic to the Chow group of zero cycles in general.  A natural byproduct of their question is whether (or, perhaps, when) vanishing of the top Chern class is sufficient to guarantee that $\mathcal{E}$ splits off a free rank $1$ summand. In view of Theorem \ref{thm:main2}, the sufficiency of the vanishing of the top Chern class is equivalent to all the higher obstructions vanishing, which from our point of view seems rather unlikely.  Nevertheless, Bhatwadekar, Das and Mandal have shown that when $k = \real$, there are situations when vanishing of the top Chern class is sufficient to guarantee splitting \cite[Theorem 4.30]{BhatwadekarDasMandal06}.

\begin{remintro}
Throughout this paper, we will assume that $k$ has characteristic unequal to $2$, but a result can be established if $k$ has characteristic $2$ as well.  Indeed, one can first establish a much stronger version of Corollary \ref{corintro:finitelymanyobstructions}.  More precisely, suppose $k$ is a perfect field having characteristic $2$.  If $X$ is a smooth $k$-scheme of dimension $d$, and $\xi: \mathcal{E} \to X$ is a rank $d$ vector bundle on $X$, then $e(\xi) = 0$ if and only if $c_d(\xi) = 0$.  Establishing this result requires somewhat different arguments, and we will write a complete proof elsewhere.
\end{remintro}

\subsubsection*{Preliminaries}
When mentioning motivic cohomology, we will assume $k$ is perfect.  Thus, for simplicity, the reader can assume that $k$ is perfect and has characteristic unequal to $2$ throughout the paper.  The proof of Theorem \ref{thm:main} in positive characteristic depends on the main result of the preprint \cite{HKO}, which, at the time of writing, depends on several other pieces of work that are still only available in preprint form.  We refer the reader to \cite{FaselChowWitt} for results regarding Chow-Witt theory, \cite{MVW} for general properties of motivic cohomology, and \cite{MV} for results about ${\mathbb A}^1$-homotopy theory.

We will consider cohomology of strictly $\aone$-invariant sheaves on a smooth scheme $X$ (see Section \ref{sec:sheaves} for some recollections about the sheaves considered in this paper).  In the introduction, we considered these sheaves on the small Nisnevich site of $X$, but below we will consider only sheaves in the Zariski topology. By, e.g., \cite[Corollary 5.43]{MField} the cohomology of a strictly $\aone$-invariant sheaf computed in the Zariski topology coincides with cohomology computed in the Nisnevich topology.

\subsubsection*{Acknowledgements}
We thank Burt Totaro for a discussion related to the proof of Theorem \ref{thm:main}. We would also like to thank the referees for their thorough reading of the first version of this paper and a number of useful remarks.

\section{A modification of the Pardon spectral sequence}
In this section, we recall the definition of twisted Milnor-Witt K-theory sheaves and various relatives. We then describe a standard filtration on twisted Milnor-Witt K-theory sheaves and analyze the associated spectral sequence.

\subsection{Unramified powers of the fundamental ideal and related sheaves}
\label{sec:sheaves}
Let $k$ be a field of characteristic different from $2$ and let $\mathrm{Sm}_k$ be the category of schemes that are separated, smooth and have finite type over $\Spec (k)$.  Let $\mathbf{W}$ be the (Zariski) sheaf on $\mathrm{Sm}_k$ associated with the presheaf $X\mapsto W(X)$, where $W(X)$ is the Witt group of $X$ (\cite{Knebusch76}, \cite{Knus91}). If $X$ is a smooth connected $k$-scheme, then the restriction of ${\mathbf W}$ to the small Zariski site of $X$ admits an explicit flasque resolution, the so called Gersten-Witt complex $C(X,\mathbf{W})$ (\cite{Balmer02}, \cite{Balmer02b}):
\[
\xymatrix@C=1.2em{W(k(X))\ar[r] & \displaystyle{\bigoplus_{x\in X^{(1)}} W_{fl}(k(x))}\ar[r]^-{d_1} &\displaystyle{\bigoplus_{x\in X^{(2)}} W_{fl}(k(x)) }\ar[r]^-{d_2} & \displaystyle{\bigoplus_{x\in X^{(3)}} W_{fl}(k(x)) }\ar[r] & \ldots.
}
\]
Here, $W_{fl}(k(x))$ denotes the Witt group of finite length $\mathcal{O}_{X,x}$-modules (\cite{Pardon82},\cite{Barge87}), which is a free $W(k(x))$-module of rank one.

For any $n\in \Z$, let $I^n(k(x))\subset W(k(x))$ be the $n$-th power of the fundamental ideal (with the convention that $I^n(k(x))=W(k(x))$ if $n\leq 0$) and let $I^n_{fl}(k(x)):=I^n(k(x))\cdot W_{fl}(k(x))$. The differentials $d_i$ of the Gersten-Witt complex respect the subgroups $I_{fl}^n(k(x))$ in the sense that $d_i(I_{fl}^n(k(x)))\subset I^{n-1}_{fl}(k(y))$ for any $i\in\N$, $x\in X^{(i)}$, $y\in X^{(i+1)}$ and $n\in\Z$ (\cite{Gille07b},\cite[Lemme 9.2.3]{FaselChowWitt}).  This yields a Gersten-Witt complex $C(X,\mathbf{I}^j)$:
\[
\xymatrix@C=1.2em{I^j(k(X))\ar[r] & \displaystyle{\bigoplus_{x\in X^{(1)}} I^{j-1}_{fl}(k(x))}\ar[r]^-{d_1} &\displaystyle{\bigoplus_{x\in X^{(2)}} I^{j-2}_{fl}(k(x)) }\ar[r] & \displaystyle{\bigoplus_{x\in X^{(3)}} I^{j-2}_{fl}(k(x)) }\ar[r] & \ldots
}
\]
for any $j\in\Z$ which provides a flasque resolution of the sheaf $\mathbf{I}^j$, i.e., the sheaf associated with the presheaf $X\mapsto H^0(C(X,\mathbf{I}^j))$.  There is an induced filtration of the sheaf $\mathbf{W}$ by subsheaves of the form:
\[
\ldots\subset \mathbf{I}^j\subset \mathbf{I}^{j-1}\subset\ldots\subset \mathbf{I}\subset \mathbf{W};
\]
the successive quotients are usually given special notation: $\overline{\mathbf{I}}^j:=\mathbf{I}^j/\mathbf{I}^{j+1}$ for any $j\in\N$.

The exact sequence of sheaves
\[
0 \longrightarrow \mathbf{I}^{j+1} \longrightarrow \mathbf{I}^j \longrightarrow \overline{\mathbf{I}}^j \longrightarrow 0
\]
yields an associated flasque resolution of $\overline{\mathbf{I}}^j$ by complexes $C(X,\overline{\mathbf{I}}^j)$ \cite[proof of Theorem 3.24]{Fasel07} of the form:
\[
\xymatrix@C=1.2em{\overline{I}^j(k(X))\ar[r] & \displaystyle{\bigoplus_{x\in X^{(1)}} \overline{I}^{j-1}(k(x))}\ar[r]^-{d_1} &\displaystyle{\bigoplus_{x\in X^{(2)}} \overline{I}^{j-2}(k(x)) }\ar[r] & \displaystyle{\bigoplus_{x\in X^{(3)}} \overline{I}^{j-2}(k(x)) }\ar[r] & \ldots.
}
\]
The subscript $fl$ appearing in the notation above has been dropped in view of the canonical isomorphism
\[
\overline{I}^j(k(x)):=I^j(k(x))/I^{j+1}(k(x)) \longrightarrow  I^j_{fl}(k(x))/I^{j+1}_{fl}(k(x))=:\overline{I}^j_{fl}(k(x))
\]
induced by any choice of a generator of $W_{fl}(k(x))$ as $W(k(x))$-module (\cite[Lemme E.1.3, Proposition E.2.1]{FaselChowWitt}).

Suppose now that $X$ is a smooth $k$-scheme and $\mathcal L$ is a line bundle on $X$.  One may define the sheaf $\mathbf{W}(\mathcal L)$ on the category of smooth schemes over $X$ as the sheaf associated with the presheaf $\{f:Y\to X\}\to W(Y,f^*\mathcal L)$, where the latter is the Witt group of the exact category of coherent locally free $\mathcal O_X$-modules equipped with the duality $\mathrm{Hom}_{\mathcal O_X}(\_,\mathcal L)$.  The constructions above extend to this ``twisted" context and we obtain sheaves $\mathbf{I}^j(\mathcal L)$ for any $j\in\Z$ and flasque resolutions of these sheaves by complexes that will be denoted $C(X,\mathbf{I}^j(\mathcal L))$.

There are canonical isomorphisms $\overline {\mathbf I}^j=\mathbf{I}^j(\mathcal L)/\mathbf{I}^{j+1}(\mathcal L)$ and we thus obtain a filtration $\ldots\subset \mathbf{I}^j(\mathcal L)\subset \mathbf{I}^{j-1}(\mathcal L)\subset\ldots\subset \mathbf{I}(\mathcal L)\subset \mathbf{W}(\mathcal L)$ and long exact sequences
\begin{equation}\label{equ:Ij}
0 \longrightarrow \mathbf{I}^{j+1}(\mathcal L) \longrightarrow \mathbf{I}^j(\mathcal L) \longrightarrow \overline{\mathbf{I}}^j \longrightarrow 0.
\end{equation}

Let $\mathcal{F}_k$ be the class of finitely generated field extensions of $k$. As usual, write $K_n^{\Mi}(F)$ for the $n$-th Milnor $K$-theory group as defined in \cite{Milnor69} (with the convention that $K_n^{\Mi}(F)=0$ if $n<0$).   The assignment $F \mapsto K_n^{\Mi}(F)$ defines a cycle module in the sense of \cite[Definition 2.1]{Rost96}. We denote by $\K_n^{\Mi}$ the associated Zariski sheaf (\cite[Corollary 6.5]{Rost96}), which has an explicit Gersten resolution by flasque sheaves (\cite[Theorem 6.1]{Rost96}). The same ideas apply for Milnor $K$-theory modulo some integer and, in particular, we obtain a sheaf $\K_n^{\Mi}/2$.

For any $F\in\mathcal{F}_k$ and any $n\in\N$, there is a surjective homomorphism $s_n:K_n^{\Mi}(F)/2\to \overline I^n(F)$ which, by the affirmation of the Milnor conjecture on quadratic forms \cite{OVV}, is an isomorphism.  The homomorphisms $s_n$ respect residue homomorphisms with respect to discrete valuations (e.g. \cite[Proposition 10.2.5]{FaselChowWitt}) and thus induce isomorphisms of sheaves $\K_n^{\Mi}/2\to \overline{\mathbf{I}}^n$ for any $n\in\N$.

For any $n\in\Z$, the $n$-th Milnor-Witt $K$-theory sheaf $\K_n^{\MW}$ can (and will) be defined as the fiber product
\[
\xymatrix{\K_n^{\MW}\ar[r]\ar[d] & \mathbf{I}^n\ar[d] \\
\K_n^{\Mi}\ar[r] & \overline{\mathbf I}^n}
\]
where the bottom horizontal morphism is the composite $\K_n^{\Mi}\to \K_n^{\Mi}/2\stackrel {s_n}\to \overline{\mathbf I}^n$ and the right-hand vertical morphism is the quotient morphism. It follows from \cite[Th\'eor\`eme 5.3]{Morel04} that this definition coincides with the one given in \cite[\S 3.2]{MField}.

If $\mathcal L$ is a line bundle on some smooth scheme $X$, then we define the $\L$-twisted sheaf $\K_n^{\MW}(\mathcal L)$ on the small Zariski site of $X$ analogously using $\L$-twisted powers of the fundamental ideal. Again, the resulting sheaf has an explicit flasque resolution obtained by taking the fiber products of the flasque resolutions mentioned above (\cite[Theorem 3.26]{Fasel07}), or by using the Rost-Schmid complex of \cite[\S 5]{MField}.  The above fiber product square yields a commutative diagram of short exact sequences of the following form:
\begin{equation}\label{equ:commutative}
\xymatrix{0\ar[r] & \mathbf{I}^{n+1}(\mathcal L)\ar[r]\ar@{=}[d] & \K_n^{\MW}(\mathcal L)\ar[r]\ar[d] & \K_n^{\Mi}\ar[r]\ar[d] & 0 \\
0\ar[r] & \mathbf{I}^{n+1}(\mathcal L)\ar[r] & \mathbf{I}^n(\mathcal L)\ar[r] & \overline{\mathbf{I}}^n\ar[r] & 0.}
\end{equation}

\subsection{The Pardon spectral sequence}\label{sec:Pardon}
Continuing to assume $k$ is a field having characteristic unequal to $2$, let $X$ be a smooth $k$-scheme and suppose $\mathcal L$ is a line bundle over $X$.   The filtration
\[
\ldots\subset \mathbf{I}^j(\mathcal L)\subset \mathbf{I}^{j-1}(\mathcal L)\subset\ldots\subset \mathbf{I}(\mathcal L)\subset \mathbf{W}(\mathcal L)
\]
yields a spectral sequence that we will refer to as the \emph{Pardon spectral sequence}.  We record the main properties of this spectral sequence here, following the formulation of \cite[Theorem 1.1]{Totaro}.

\begin{thm}\label{thm:pardon}
Assume $k$ is a field having characteristic unequal to $2$, $X$ is a smooth $k$-scheme, and $\L$ is a line bundle on $X$.  There exists a spectral sequence $E(\mathcal L)_2^{p,q}=H^p(X,\overline{\mathbf{I}}^q)\Rightarrow H^p(X,\mathbf{W}(\mathcal L))$. The differentials $d({\mathcal L})_r$ are of bidegree $(1,r-1)$ for $r\geq 2$, and the groups $H^p(X,\overline{\mathbf I}^q)$ are trivial unless $0\leq p\leq q$.  There are identifications $H^p(X,\overline{\mathbf I}^p)=CH^p(X)/2$ and the differential $d_2^{pp}:H^p(X,\sI^p)\to H^{p+1}(X,\sI^{p+1})$ coincides with the Steenrod square operation $Sq^2$ as defined by Voevodsky (\cite{VRed}) and Brosnan (\cite{Brosnan03}) when $\mathcal L$ is trivial. Finally, if $k$ has finite $2$-cohomological dimension, the spectral sequence is bounded.
\end{thm}

\begin{proof}
All the statements are proved in \cite[proof of Theorem 1.1]{Totaro} except the last one, which follows from the cohomology vanishing statement contained in \cite[Proposition 5.1]{AsokFaselThreefolds}.
\end{proof}

\begin{rem}
We will describe the differential $d(\mathcal L)_2^{pp}:H^p(X,\sI^p)\to H^{p+1}(X,\sI^{p+1})$ for $\mathcal L$ nontrivial in Theorem \ref{thm:twistplus}.
\end{rem}

Since $\mathbf{W}(\L) = \mathbf{I}^0(\L)$ by convention, truncating the above filtration allows us to construct a spectral sequence abutting to the cohomology of $\mathbf{I}^j(\L)$ for arbitrary $j \geq 0$:
\[
\ldots\subset \mathbf{I}^{n+j}(\mathcal L)\subset \mathbf{I}^{n+j-1}(\mathcal L)\subset\ldots\subset \mathbf{I}^{j+1}(\mathcal L)\subset \mathbf{I}^j(\mathcal L).
\]
The resulting spectral sequence $E(\mathcal L,j)^{p,q}$ is very similar to the Pardon spectral sequence. Indeed, $E(\mathcal L,j)_2^{p,q}=0$ if $q<j$ and $E(\mathcal L,j)_2^{p,q}=E(\mathcal L)_2^{p,q}$ otherwise.  Similarly $d(\mathcal L,j)_2^{p,q}=0$ if $q<j$ and $d(\mathcal L,j)_2^{p,q}=d(\mathcal L)_2^{p,q}$ otherwise.  We call this spectral sequence the \emph{$j$-truncated Pardon spectral sequence} and it will be one of the main objects of study in this paper.  Using the description of the $E_2$-page of this spectral sequence and the associated differentials, the proof of the following lemma is straightforward (and left to the reader).

\begin{lem}\label{lem:truncatedss}
Assume $k$ is a field having characteristic unequal to $2$ and suppose $X$ is a smooth $k$-scheme of dimension $d$.  There are identifications $E(\mathcal L,d)^{d,d}_{\infty}=CH^d(X)/2$ and, for any $n\geq 1$, $E(\mathcal L,d)^{d,d+n}_{m}=E(\mathcal L)^{d,d+n}_m$ if $m\leq n+1$ and exact sequences
\[
\xymatrix@C=4em{E(\mathcal L)_{n+1}^{d-1,d}\ar[r]^{d(\mathcal L)_{n+1}^{d-1,d}} & E(\mathcal L)_{n+1}^{d,d+n}\ar[r] & E(\mathcal L,d)_{\infty}^{d,d+n}\ar[r] & 0.}
\]
\end{lem}

Using the monomorphism $\mathbf{I}^{j+1}(\mathcal L)\subset \K_j^{\MW}(\mathcal L)$ described in the previous section, we can consider the filtration of $\mathbf{I}^{j+1}(\mathcal L)$ as a filtration of $\K_j^{\MW}(\mathcal L)$ of the form:
\[
\ldots\subset \mathbf{I}^{n+j}(\mathcal L)\subset \mathbf{I}^{n+j-1}(\mathcal L)\subset\ldots\subset \mathbf{I}^{j+1}(\mathcal L)\subset \K_j^{\MW}(\mathcal L).
\]
Once again, the spectral sequence $E(\mathcal L,\MW)^{p,q}$ associated with this filtration is very similar to the $j$-truncated Pardon spectral sequence.  Indeed, there are identifications $E(\mathcal L,\MW)_2^{p,q}=E(\mathcal L,j)_2^{p,q}$ if $q\neq j$ and $E(\mathcal L,\MW)_2^{p,j}=H^p(X,\K_j^{\Mi})$.  In order to describe the terms $E(\mathcal L,\MW)_{\infty}^{j,q}$ in the situation of interest, we first need a few definitions.

Consider the commutative diagram of sheaves with exact rows from Diagram \ref{equ:commutative}
\[
\xymatrix@C=2em{0\ar[r] & \mathbf{I}^{j+1}(\mathcal L)\ar[r]\ar@{=}[d] & \K_j^{\MW}(\mathcal L)\ar[r]\ar[d] & \K_j^{\Mi}\ar[r]\ar[d] & 0 \\
0\ar[r] & \mathbf{I}^{j+1}(\mathcal L)\ar[r] & \mathbf{I}^{j}(\mathcal L)\ar[r] & \sI^j\ar[r] & 0.}
\]
The right vertical homomorphism $\K^{\Mi}_j \to \sI^j$ is described in the previous subsection and yields, in particular, a homomorphism $H^{j-1}(X,\K_j^{\Mi})\to H^{j-1}(X,\sI^j)$ whose image we denote by $G_2(j)$.  Now, $H^{j-1}(X,\sI^j)=E(\mathcal L,j)_2^{j-1,j}=E(\mathcal L)_2^{j-1,j}$ and there is a differential
\[
d(\mathcal L)_2^{j-1,j}:E(\mathcal L)_2^{j-1,j} \longrightarrow E(\mathcal L)_2^{j,j+1}.
\]
We set $G_3(j):=G_2(j)\cap \mathrm{ker}(d(\mathcal L)_2^{j-1,j})$ and write $\overline G_3(j)$ for its image in $E(\mathcal L)_3^{j-1,j}$. There is also a differential
\[
d(\mathcal L)_3^{j-1,j}:E(\mathcal L)_3^{j-1,j} \longrightarrow E(\mathcal L)_3^{j,j+2}
\]
and we set $G_4(j):=\overline G_3(j)\cap \mathrm{ker}(d(\mathcal L)_3^{j-1,j})$ and define $\overline G_4(j)$ to be its image in $E(\mathcal L)_4^{j-1,j}$.
Continuing inductively, we can define a sequence of subgroups $\overline G_{n}(j)\subset E(\mathcal L)_n^{j-1,j}$ for any $n\geq 2$.

\begin{lem}\label{lem:KMWss}
If $k$ is a field having characteristic unequal to $2$, and $X$ is a smooth $k$-scheme of dimension $d$, then there are isomorphisms $E(\mathcal L,\MW)^{d,d}_{\infty}=CH^d(X)$, and $E(\mathcal L,\MW)^{d-1,d}_2=H^{d-1}(X,\K_d^{M})$. Furthermore, for any integer $n\geq 1$, there are identifications $E(\mathcal L,\MW)^{d,d+n}_{m}=E(\mathcal L)^{d,d+n}_m$ if $m\leq n+1$ and exact sequences of the form
\[
\xymatrix@C=4em{\overline G_{n+1}(d)\ar[r]^-{d(\mathcal L)_{n+1}^{d-1,d}} & E(\mathcal L)^{d,d+n}_{n+1}\ar[r] & E(\mathcal L,\MW)^{d,d+n}_{\infty}\ar[r] & 0.}
\]
\end{lem}

\begin{proof}
The morphism of sheaves $\K^{\MW}_d(\mathcal L)\to \mathbf{I}^d(\mathcal L)$ is compatible with the filtrations:
\[
\xymatrix{\ldots\ar[r] & \mathbf{I}^{d+n}(\mathcal L)\ar[r]\ar@{=}[d] & \mathbf{I}^{d+n-1}(\mathcal L)\ar[r]\ar@{=}[d] & \ldots \ar[r] & \mathbf{I}^{d+1}(\mathcal L)\ar@{=}[d]\ar[r] & \K^{\MW}_d(\mathcal L)\ar[d] \\
\ldots\ar[r] & \mathbf{I}^{d+n}(\mathcal L)\ar[r] & \mathbf{I}^{d+n-1}(\mathcal L)\ar[r] & \ldots \ar[r] & \mathbf{I}^{d+1}(\mathcal L)\ar[r] & \mathbf{I}^d(\mathcal L)}
\]
In particular, the induced maps of quotient sheaves are simply the identity map, except at the last spot where they fit into the commutative diagram
\[
\xymatrix@C=2em{0\ar[r] & \mathbf{I}^{d+1}(\mathcal L)\ar[r]\ar@{=}[d] & \K_d^{\MW}(\mathcal L)\ar[r]\ar[d] & \K_d^{\Mi}\ar[r]\ar[d] & 0 \\
0\ar[r] & \mathbf{I}^{d+1}(\mathcal L)\ar[r] & \mathbf{I}^{d}(\mathcal L)\ar[r] & \sI^d\ar[r] & 0}
\]
The result now follows from the definition of the groups $\overline G_i(d)$ and Lemma \ref{lem:truncatedss}.
\end{proof}

\begin{rem}\label{rem:2torsion}
By construction, there are epimorphisms $E(\mathcal L,\MW)_\infty^{d,d+n}\to E(\mathcal L,d)_\infty^{d,d+n}$ for any $n\geq 0$. Indeed, $\overline G_{n+1}(d)$ is, by definition, a subgroup of $E(\mathcal L)_{n+1}^{d-1,d}$ and the diagram
\[
\xymatrix{\overline G_{n+1}(d)\ar[r]\ar[d] & E(\mathcal L)^{d,d+n}_{n+1}\ar@{=}[d] \ar[r]& E(\mathcal L,\MW)_\infty^{d,d+n}\ar[r]\ar@{-->}[d] & 0 \\
E(\mathcal L)_{n+1}^{d-1,d}\ar[r] & E(\mathcal L)^{d,d+n}_{n+1}\ar[r] & E(\mathcal L,d)_\infty^{d,d+n}\ar[r] & 0}
\]
commutes. 

Suppose that $X$ is a smooth $k$-scheme of dimension $d$ such that the Chow group of $0$-cycles $CH^d(X)$ is $2$-torsion free. In that case, we claim that the dotted arrow in the above diagram is an isomorphism. To see this, observe that the exact sequence of sheaves
\[
0 \longrightarrow 2\K_d^{\Mi} \longrightarrow \K_d^{\Mi} \longrightarrow \K_d^{\Mi}/2 \longrightarrow 0
\]
yields an exact sequence
\[
\xymatrix@C=1.2em{H^{d-1}(X,\K_d^{\Mi}) \ar[r] & H^{d-1}(X,\K_d^{\Mi}/2) \ar[r] & H^{d}(X,2\K_d^{\Mi}) \ar[r] & H^{d}(X,\K_d^{\Mi}) \ar[r] & H^d(X,\K_d^{\Mi}/2) \ar[r] & 0.}
\]
The epimorphism $\K_d^{\Mi}\stackrel 2\to 2\K_d^{\Mi}$ yields an isomorphism $H^{d}(X,\K_d^{\Mi})\to H^{d}(X,2\K_d^{\Mi})$ and we deduce the following exact sequence from Rost's formula and the definition of $G_2(d)$:
\[
0 \longrightarrow G_2(d) \longrightarrow H^{d-1}(X,\K_d^{\Mi}/2) \longrightarrow  CH^d(X)\stackrel 2 \longrightarrow CH^d(X) \longrightarrow CH^d(X)/2 \longrightarrow 0.
\]
Since $CH^d(X)$ is $2$-torsion free, it follows that $G_2(d)=H^{d-1}(X,\K_d^{\Mi}/2)$ and by inspection we obtain an identification $\overline G_{n+1}(d)=E(\mathcal L)_{n+1}^{d-1,d}$.  We therefore conclude that the dotted arrow in the above diagram is an isomorphism.
\end{rem}

\begin{thm}\label{thm:obstructions}
Suppose $k$ is a field having characteristic unequal to $2$ and finite $2$-cohomological dimension, $X$ is a smooth $k$-scheme of dimension $d$ and $\mathcal L$ is a line bundle over $X$. For any $\alpha\in H^d(X,\K_d^{\MW}(\mathcal L))$ there are inductively defined obstructions $\Psi^n(\alpha)\in E(\mathcal L,\MW)_{\infty}^{d,d+n}$ for $n\geq 0$ such that $\alpha=0$ if and only if $\Psi^n(\alpha)=0$ for any $n\geq 0$.
\end{thm}

\begin{proof}
The filtration
\[
\ldots\subset \mathbf{I}^{n+d}(\mathcal L)\subset \mathbf{I}^{n+d-1}(\mathcal L)\subset\ldots\subset \mathbf{I}^{d+1}(\mathcal L)\subset \K_d^{\MW}(\mathcal L)
\]
to which the spectral sequence $E(\mathcal L,\MW)^{p,q}$ is associated yields a filtration $F^nH^d(X,\K_d^{\MW}(\mathcal L))$ for $n\geq 0$ of the cohomology group $H^d(X,\K_d^{\MW}(\mathcal L))$ with $F^0H^d(X,\K_d^{\MW}(\mathcal L))=H^d(X,\K_d^{\MW}(\mathcal L))$ and
\[
F^nH^d(X,\K_d^{\MW}(\mathcal L))=\mathrm{Im}(H^d(X,\mathbf{I}^{d+n}(\mathcal L)) \longrightarrow H^d(X,\K_d^{\MW}(\mathcal L)))
\]
for $n\geq 1$. Further, $F^nH^d(X,\K_d^{\MW}(\mathcal L))/F^{n+1}H^d(X,\K_d^{\MW}(\mathcal L)):=E(\mathcal L,\MW)^{d,d+n}_\infty$ and the cohomological vanishing statement of \cite[Proposition 5.1]{AsokFaselThreefolds} implies that only finitely many of the groups appearing above can be non-trivial. If we define the obstructions $\Psi^n(\alpha)$ to be the image of $\alpha$ in the successive quotients, the result is clear.
\end{proof}

The above result gives an inductively defined sequence of obstructions to decide whether an element of $H^d(X,\K_d^{\MW}(\mathcal L))$ is trivial.  Our next goal is to provide a ``concrete" description of the differentials appearing in the spectral sequence.  Lemmas \ref{lem:truncatedss} and \ref{lem:KMWss} imply that these differentials are essentially the differentials in the Pardon spectral sequence, and it is for that reason that we focus on the latter in the remaining sections.

\section{Some properties of the differentials}
In this section, we establish some properties of the differentials in the Pardon spectral sequence and thus the spectral sequence constructed in the previous section abutting to cohomology of twisted Milnor-Witt K-theory sheaves.  We first recall how these differentials are defined and then show that, essentially, they can be viewed as bi-stable operations in motivic cohomology.

\subsection{The operation $\Phi_{i,j}$}\label{ss:phiij}
Suppose $X$ is a smooth $k$-scheme and $\L$ is a line bundle on $X$.  Recall that for any $j\in \N$, the sheaf $\mathbf{I}^j(\L)$ comes equipped with a reduction map $\mathbf{I}^j(\L) \to \bar{\mathbf{I}}^j$ and that there is a canonical isomorphism $\K^{\Mi}_j/2 \to \bar{\mathbf{I}}^j$; we use this identification without mention in the sequel.  The exact sequence
\[
0 \longrightarrow \mathbf{I}^{j+1}(\L) \longrightarrow \mathbf{I}^j(\L) \longrightarrow \bar{\mathbf{I}}^j \longrightarrow 0
\]
yields a connecting homomorphism
\[
H^i(X,\bar{\mathbf{I}}^j) \stackrel{\partial_{\L}}{\longrightarrow} \H^{i+1}(X,\mathbf{I}^{j+1}(\L)).
\]
The reduction map gives a homomorphism
\[
H^{i+1}(X,\mathbf{I}^{j+1}(\L)) \longrightarrow H^{i+1}(X,\bar{\mathbf{I}}^{j+1}).
\]
Taking the composite of these two maps yields a homomorphism that is precisely the differential $d(\mathcal L)_2^{i,j}$. We state the following definition in order to avoid heavy notation.

\begin{defn}
\label{defn:phiijL}
If $X$ is a smooth scheme, and $\L$ is a line bundle on $X$, write
\[
\Phi_{i,j,\L}: H^i(X,\bar{\mathbf{I}}^j) \longrightarrow H^{i+1}(X,\bar{\mathbf{I}}^{j+1}).
\]
for the composite of the connecting homomorphism $\partial_{L}$ and the reduction map just described.  If $\L$ is trivial, suppress it from the notation and write $\Phi_{i,j}$ for the resulting homomorphism. Anticipating Theorem \ref{thm:main}, we sometimes refer to $\Phi_{i,j,\mathcal L}$ as an operation.
\end{defn}

When $i = j$, via the identification $\bar{\mathbf{I}}^j \cong \K^{\Mi}_j/2$, the map $\Phi_{i,i}$ can be viewed as a morphism $Ch^i(X) \to Ch^{i+1}(X)$, where $Ch^i(X) = CH^i(X)/2$.  As stated in Theorem \ref{thm:pardon}, Totaro identified this homomorphism as $Sq^2$.  More generally, we observe that the homomorphisms $\Phi_{i,j,\L}$ are functorial with respect to pull-backs by definition.

\subsection{Bi-stability of the operations $\Phi_{i,j}$}
\label{ss:bistability}
We now study bi-stability, i.e., stability with respect to $\pone$-suspension, of the operations $\Phi_{i,j}$.  If $X$ is a smooth scheme, we then need to compare an operation on $X$ and a corresponding operation on the space $X_+ \sma \pone$. The reader unfamiliar to this notation can take the following ad hoc definition. If $\mathbf{F}$ is a sheaf, then $H^i(X_+ \sma \pone,\mathbf{F})$ is defined to be the cokernel of the pull-back homomorphism
\[
H^i(X,\mathbf{F}) \longrightarrow H^i(X\times \pone,\mathbf{F}).
\]
In case $\mathbf{F}= \bar {\mathbf I}^j$, we use the projective bundle formula in $\bar{\mathbf{I}}^j$-cohomology (see, e.g., \cite[\S 4]{FaselIJ}) to identify this group in terms of cohomology on $X$. Indeed, we have an identification
\[
H^i(X\times \pone,\bar {\mathbf I}^j) \cong H^i(X,\bar {\mathbf I}^j)\oplus H^{i-1}(X,\bar {\mathbf I}^{j-1})\cdot \bar c_1(\mathcal O(-1)),
\]
where $\bar c_1(\mathcal O(-1))$ is the first Chern class of $\mathcal O(-1)$ in $H^1(X,\K_1^{\Mi}/2)=CH^1(X)/2$. Unwinding the definitions, this corresponds to an isomorphism of the form
\[
H^i(X_+ \sma \pone,\bar{\mathbf{I}}^j) \cong H^{i-1}(X,\bar {\mathbf I}^{j-1})
\]
that is functorial in $X$.  Using this isomorphism, we can compare the operation $\Phi_{i,j}$ on $H^i(X_+ \sma \pone,\bar{\mathbf{I}}^j)$ with the operation $\Phi_{i-1,j-1}$ on $H^{i-1}(X,\bar {\mathbf I}^{j-1})$.

\begin{prop}\label{prop:bistable}
There is a commutative diagram of the form
\[
\xymatrix{
H^{i}(X_+ \sma \pone,\bar{\mathbf{I}}^j) \ar[r]^-{\Phi_{i,j}}\ar[d] & H^{i+1}(X_+ \sma \pone,\bar{\mathbf{I}}^{j+1}) \ar[d] \\
H^{i-1}(X,\bar{\mathbf{I}}^{j-1}) \ar[r]_-{\Phi_{i-1,j-1}} & H^{i}(X,\bar{\mathbf{I}}^j),
}
\]
where the vertical maps are the isomorphisms described before the statement.  
\end{prop}

\begin{proof}
The operation $\Phi_{i,j}$ is induced by the composite morphism of the connecting homomorphism associated with the short exact sequence
\[
0 \longrightarrow \mathbf{I}^{j+1} \longrightarrow \mathbf{I}^{j} \longrightarrow \bar{\mathbf{I}}^{j} \longrightarrow 0
\]
and the reduction map $\mathbf{I}^{j+1} \to \bar{\mathbf{I}}^{j+1}$.  The contractions of $\mathbf{I}^j$ and $\bar{\mathbf{I}}^j$ are computed in \cite[Lemma 2.7 and Proposition 2.8]{AsokFaselSpheres} and our result follows immediately from the proofs of those statements.
\end{proof}

\begin{rem}
Because of the above result, we will abuse terminology and refer to $\Phi_{i,j}$ as a bi-stable operation.
\end{rem}

\subsection{Non-triviality of the operation $\Phi_{i-1,i,\L}$}
\label{ss:nontriviality}
Our goal in this section is to prove that the operation $\Phi_{i-1,i}$ is nontrivial.  By definition, the operation $\Phi_{i-1,i}$ can be computed as follows: given an element $\alpha\in H^{i-1}(X,\bar {\mathbf I}^i)$, we choose a lift to $C^{i-1}(X, \mathbf{I}^i)$, apply the boundary homomorphism to obtain an element $d_{i-1}(\alpha)\in C^{i}(X, \mathbf{I}^{i})$ which becomes trivial under the homomorphism $C^{i}(X, \mathbf{I}^{i})\to C^{i}(X, \bar {\mathbf{I}}^{i})$ (since $\alpha$ is a cycle). There exists thus a unique lift of $d_{i-1}(\alpha)\in C^{i}(X, \mathbf{I}^{i+1})$, which is a cycle since $d_id_{i-1}=0$. Its reduction in $H^i(X, \bar {\mathbf I}^{i})$ is $\Phi_{i-1,i}(\alpha)$ by definition.  We use the identification $H^{i-1}(X,\bar {\mathbf I}^i) \cong H^{i-1}(X,\K^{\Mi}_i/2)$ and the computations of Suslin in the case where $X = SL_3$ to provide explicit generators.  More precisely, \cite[Theorem 2.7]{Suslin91} shows that $H^1(SL_3,\K^{\Mi}_2/2)=\Z/2$, $H^2(SL_3,\K^{\Mi}_3/2)=\Z/2$.  We begin by finding explicit generators of the groups considered by Suslin and transfer those generators under the isomorphisms just described to obtain explicit representatives of classes in $H^1(SL_3,\bar {\mathbf I}^2)$ and $H^2(SL_3,\bar {\mathbf{I}}^3)$.  Then, we explicitly compute the connecting homomorphism and the reduction.  Our method and notation will follow closely \cite[\S 2]{Suslin91}.

For any $n\in \N$, let $Q_{2n-1}\subset \A^{2n}$ be the hypersurface given by the equation $\sum_{i=1}^n x_iy_i=1$. Let $SL_n=\mathrm{Spec}(k[(t_{ij})_{1\leq i,j\leq n}]/\langle \det (t_{ij})-1\rangle)$ and write $\alpha_n=(t_{ij})_{1\leq i,j\leq n}$ for the universal matrix on $SL_n$, and $(t^{ij})_{1\leq i,j\leq n}$ for its inverse $\alpha_n^{-1}$. For $n\geq 2$, we embed $SL_{n-1}$ into $SL_n$ as usual by mapping a matrix $M$ to $\mathrm{diag}(1,M)$, and we observe that the quotient is precisely $Q_{2n-1}$ by means of the homomorphism $f:SL_n\to Q_{2n-1}$ given by $f^*(x_i)=t_{1i}$ and $f^*(y_i)=t^{i1}$. Now $Q_{2n-1}$ is covered by the affine open subschemes $U_i:=D(x_i)$ and the projection $f:SL_n\to Q_{2n-1}$ splits over each $U_i$ by means of a matrix $\gamma_i\in E_n(U_i)$ given for instance in \cite[\S 2]{Suslin91}. The only properties that we will use here are that these sections induce isomorphisms $f^{-1}(U_i)\simeq U_i\times SL_{n-1}$ mapping $(\alpha_n)_{\vert f^{-1}(U_i)}\gamma_i^{-1}$ to $\mathrm{diag}(1,\alpha_{n-1})$. Recall next from \cite[\S 2]{Gillet81}, that one can define Chern classes
\[
c_i:K_1(X) \longrightarrow H^{i}(X,\K_{i+1}^{\Mi}/2)
\]
functorially in $X$. In particular, we have Chern classes $c_i:K_1(SL_n)\to H^i(SL_n,\K_{i+1}^{\Mi}/2)$ and we set $d_{i,n}:=c_i(\alpha_n)$.

The stage being set, we now proceed to our computations. We will implicitly use the Gersten resolution of the sheaves $\K^{\Mi}_i/2$ in our computations below. Observe first that the equations $x_2=\ldots=x_n=0$ define an integral subscheme $Z_n\subset Q_{2n-1}$, and that the global section $x_1$ is invertible on $Z_n$. It follows that it defines an element in $(\K_1^{\Mi}/2)(k(Z_n))$ and a cycle $\theta_n\in H^{n-1}(Q_{2n-1},\K^{\Mi}_n)$.

\begin{lem}
For any smooth scheme $X$, the $H^*(X,\K_*^{\Mi}/2)$-module $H^*(Q_{2n-1}\times X,\K_*^{\Mi}/2)$ is free with basis $1,\theta_n$.
\end{lem}

\begin{proof}
Apply the proof of \cite[Theorem 1.5]{Suslin91} mutatis mutandis.
\end{proof}

Since $Q_3=SL_2$, we can immediately deduce a basis for the cohomology of $SL_2$. However, we can reinterpret $\theta_2$ as follows.

\begin{lem}\label{lem:firstChern}
If $X$ is a smooth scheme, then $H^*(SL_2\times X,\K^{\Mi}_*/2)$ is a free $H^*(X,\K^{\Mi}_*/2)$-module generated by $1\in H^0(X,\K^{\Mi}_0/2)$ and $d_{1,2}\in H^1(SL_2,\K_2^{\Mi}/2)$.
\end{lem}

\begin{proof}
Again, this is essentially \cite[proof of Proposition 1.6]{Suslin91}.
\end{proof}

Before stating the next lemma, recall that we have a projection morphism $f:SL_3\to Q_5$, yielding a structure of $H^*(Q_5,\K^{\Mi}_*/2)$-module on the cohomology of $SL_3$.

\begin{lem}\label{lem:Chern}
The $H^*(Q_5,\K^{\Mi}_*/2)$-module $H^*(SL_3,\K^{\Mi}_*/2)$ is free with basis $1$ and $d_{1,3}$.
\end{lem}

\begin{proof}
Using Mayer-Vietoris sequences in the spirit of \cite[Lemma 2.2]{Suslin91}, we see that it suffices to check locally that $1$ and $d_{1,3}$ is a basis. Let $U_i\subset Q_{2n-1}$ be the open subschemes defined above. We know that we have an isomorphism $f^{-1}(U_i)\simeq U_i\times SL_2$ mapping $(\alpha_3)_{\vert f^{-1}(U_i)}\gamma_i^{-1}$ to $\mathrm{diag}(1,\alpha_2)$. The Chern class $c_1$ being functorial, we have a commutative diagram
\[
\xymatrix{K_1(SL_3)\ar[r]^-{c_1}\ar[d]_-{i^*} & H^1(SL_3,\K^{\Mi}_2/2)\ar[d]^-{i^*} \\
K_1(f^{-1}(U_i))\ar[r]_-{c_1} & H^1(f^{-1}(U_i),\K^{\Mi}_2/2)
}
\]
where the vertical homomorphisms are restrictions. We thus see that $i^*(d_{1,3})=i^*(c_1(\alpha_3))=c_1(i^*(\alpha_3))$. Since $\gamma_i\in E_3(U_i)$, we see that $c_1(i^*(\alpha_3))=c_1(p^*\alpha_2)=p^*d_{1,2}$ where $p:f^{-1}(U_i)\to SL_2$ is the projection. The result now follows from Lemma \ref{lem:firstChern}.
\end{proof}

Combining Lemmas \ref{lem:firstChern} and \ref{lem:Chern}, we immediately obtain the following result.

\begin{cor}\label{cor:explicit}
We have $H^1(SL_3,\K_2^{\Mi}/2)=\Z/2\cdot d_{1,3}$ and $H^2(SL_3,\K^{\Mi}_3/2)=\Z/2\cdot f^*(\theta_3)$.
\end{cor}

The cycle $f^*(\theta_3)$ is very explicit. Indeed, it can be represented by the class of the global section $t_{11}$ in $(\K_1^{\Mi}/2)(k(z_1))$ where $z_1$ is given by the equations $t_{12}=t_{13}=0$. We now make $d_{1,3}$ more explicit. Recall that $\alpha_3=(t_{ij})$ is the universal matrix on $SL_3$ and $\alpha_3^{-1}=(t^{ij})$ is its inverse. In particular, we have $\sum_{j=1}^3 t_{ij} t^{jk}=\delta_{jk}=\sum_{j=1}^3 t^{ij}t_{jk}$.

\begin{lem}\label{lem:H1K2}
If $y_1\in SL_3^{(1)}$ is defined by the ideal $\langle t^{13}\rangle$ and $y_2\in SL_3^{(1)}$ is defined by the ideal $\langle t_{12}\rangle$, then a generator for the group $H^1(SL_3,\K^{\Mi}_2/2) \cong \Z/2$, is given by the class of the symbol
\[
\xi:=\{t^{12}\} + \{ t_{13}\}
\]
in $\K_1^{\Mi}(k(y_1))/2\oplus \K_1^{\Mi}(k(y_2))/2$.
\end{lem}

\begin{proof}
The image of $\{ t_{13}\}$ under the boundary map in the Gersten complex is the generator of $\K_0^M(k(z_1))/2$ where $z_1$ is the point defined by the ideal $I_1:=\langle t_{12}, t_{13}\rangle$, while the image of $\{ t^{12}\}$ is the generator of $\K_0^M(k(z_2))/2$ where $z_2$ is the point defined by the ideal $I_2:=\langle t^{12},t^{13}\rangle$. It suffices then to check that $z_1=z_2$ to conclude that $\xi$ is a cycle.

The equality $\sum_{j=1}^3 t^{1j}t_{j1}=1$ shows that $t^{11}$ is invertible modulo $I_2$ and we deduce from $\sum_{j=1}^3 t^{1j}t_{j2}=0$ that $t_{12}\in I_2$. Similarly, we deduce from $\sum_{j=1}^3 t^{1j}t_{j3}=0$ that $t_{13}\in I_2$ and therefore $I_1\subset I_2$. Reasoning symmetrically we obtain that $I_2\subset I_1$, proving the claim.

Since $\xi$ is a cycle, it defines a class in $H^1(SL_3,\K^{\Mi}_2/2) \cong \Z/2$ and it suffices thus to show that the class of $\xi$ is non trivial to conclude.
Consider the embedding (of schemes, but not of group schemes) $g:SL_2\to SL_3$ given by
\[
\begin{pmatrix} u_{11} & u_{12} \\
u_{21} & u_{22}\end{pmatrix} \mapsto \begin{pmatrix} 0 & -1 & 0 \\ u_{11} & 0 & u_{12} \\ u_{21} & 0 & u_{22}\end{pmatrix}.
\]
Since this morphism factors through the open subscheme $SL_3[t_{12}^{-1}]=f^{-1}(U_2)$ and the inverse of the above matrix is given by the matrix
\[
\begin{pmatrix} 0 & u_{22} & -u_{12} \\ -1 & 0 & 0 \\ 0 & -u_{21} & u_{11}\end{pmatrix},
\]
it follows that $g^*(\xi)$ is represented by the class of $\{ u_{22}\}$ in $\K_1^{\Mi}(k(s))/2$, where $s$ is given by $u_{12}=0$.  One can then verify directly that this cycle equals the generator $d_{1,2}$ given in Lemma \ref{lem:firstChern}, and it follows that $\xi\neq 0$.
\end{proof}

\begin{prop}
\label{prop:nontriviality}
The operation $\Phi_{i-1,i}$ is non-trivial.
\end{prop}

\begin{proof}
We compute the effect of the operation $\Phi_{1,2}$ on elements of $H^1(SL_3,\K^{\Mi}_2/2)$. By definition, $\Phi_{1,2}$ is the composite
\[
H^1(SL_3,\K^{\Mi}_2/2)=H^1(SL_3,\bar {\mathbf I}^2)\longrightarrow H^2(SL_3,{\mathbf I}^3)\longrightarrow H^2(SL_3,\bar {\mathbf I}^3)=H^2(SL_3,\K^{\Mi}_3/2)
\]
where the left-hand map is the boundary homomorphism associated with the exact sequence of sheaves
\[
0 \longrightarrow {\mathbf I}^3 \longrightarrow {\mathbf I}^2 \longrightarrow \bar{\mathbf I}^2 \longrightarrow 0
\]
and the right-hand map is the projection associated with the morphism of sheaves ${\mathbf I}^3\to \bar {\mathbf I}^3$.  We will show that $\Phi_{1,2}$ is an isomorphism by showing that the explicit generator of $H^1(SL_3,\K^{\Mi}_2/2)$ constructed in Lemma \ref{lem:H1K2} is mapped to the explicit generator of $H^2(SL_3,\K^{\Mi}_3/2)$ constructed in Corollary \ref{cor:explicit}.

Recall from Section \ref{sec:sheaves} the Gersten resolution $C(X,{\mathbf I}^j)$ of the sheaf ${\mathbf I}^j$, which takes the form
\[
\xymatrix@C=1.2em{I^j(k(X))\ar[r] & \displaystyle{\bigoplus_{x\in X^{(1)}} I^{j-1}_{fl}(k(x))}\ar[r]^-{d_1} &\displaystyle{\bigoplus_{x\in X^{(2)}} I^{j-2}_{fl}(k(x)) }\ar[r] & \displaystyle{\bigoplus_{x\in X^{(3)}} I^{j-3}_{fl}(k(x)) }\ar[r] & \ldots
}
\]
where $X$ is a smooth scheme, and $I^{j-1}_{fl}(k(x))=I^{j-1}(k(x))\cdot W_{fl}(\mathcal O_{X,x})$.  Take $X = SL_3$.

An explicit lift of the generator of $H^1(SL_3,\K^{\Mi}_2/2)$ given in Lemma \ref{lem:H1K2} is of the form
\[
\langle -1,t^{12}\rangle\cdot \rho_1+\langle -1,t_{13}\rangle \cdot \rho_2
\]
where $\rho_1:k(y_1)\to \mathrm{Ext}^1_{\mathcal O_{X,y_1}}(k(y_1),\mathcal O_{X,y_1})$ is defined by mapping $1$ to the Koszul complex $Kos(t^{13})$ associated with the regular sequence $t^{13}$, and similarly $\rho_2:k(y_2)\to \mathrm{Ext}^1_{\mathcal O_{X,y_2}}(k(y_2),\mathcal O_{X,y_2})$ is defined by $1\mapsto Kos(t_{12})$. Using \cite[Section 3.5]{FaselChowWitt}, the boundary $d_1$ of the above generator is of the form $\nu_1+\nu_2$, where
\[
\nu_1:k(z) \longrightarrow \mathrm{Ext}^2_{\mathcal O_{X,z}}(k(z),\mathcal O_{X,z})
\]
is defined by $1\mapsto Kos(t^{13},t^{12})$ and
\[
\nu_2:k(z) \longrightarrow \mathrm{Ext}^2_{\mathcal O_{X,z}}(k(z),\mathcal O_{X,z})
\]
is defined by $1\mapsto Kos(t_{12},t_{13})$. Recall from the proof of Lemma \ref{lem:H1K2} that $t^{11}\in \mathcal{O}_{X,z}^\times$ and it follows thus from the identities $\sum_{j=1}^3 t^{1j}t_{jk}=0$ for $k=1,2$ that we have
\[
\begin{pmatrix} t_{12} \\ t_{13}\end{pmatrix}=\begin{pmatrix} -t_{32}/t^{11} & -t_{22}/t^{11}  \\  -t_{33}/t^{11} & -t_{23}/t^{11}\end{pmatrix} \begin{pmatrix} t^{13} \\ t^{12}\end{pmatrix}.
\]
Now $t_{32}t_{23}-t_{22}t_{33}=-t^{11}$
and $t^{11}t_{11}=1$ modulo $\langle t^{12},t^{13}\rangle$ and we therefore get
\[
\nu_1+\nu_2=\langle 1,t_{11}\rangle \cdot \nu_1=(\langle 1,1\rangle +\langle -1,t_{11}\rangle)\cdot \nu_1
\]
A simple computation shows that $\langle 1,1\rangle \cdot \nu_1$ is the boundary of $(\langle 1,t_{13}\rangle \otimes \langle 1,t_{13}\rangle)\cdot \rho_2$ and therefore vanishes in $H^2(SL_3,{\mathbf I}^3)$. Now the class of $\langle -1,t_{11}\rangle\cdot \nu_1$ in $H^2(SL_3,\bar {\mathbf I}^3)=H^2(SL_3,\K^{\Mi}_3/2)$ is precisely a generator as shown by Corollary \ref{cor:explicit}. Thus, $\Phi_{1,2}:H^1(SL_3,\K^{\Mi}_2/2)\to H^2(SL_3,\K^{\Mi}_3/2)$ is an isomorphism.
\end{proof}

\subsection{Identification of $\Phi_{i-1,i,\L}$}
\label{ss:identification}
If $\mathcal L$ is a line bundle over our smooth $k$-scheme $X$, we write $\overline c_1(\mathcal L)$ for its first Chern class in $H^1(X,\K_1^{\Mi}/2)=CH^1(X)/2$.

\begin{thm}\label{thm:twistplus}
\label{thm:cohomologyoperation}
For any smooth scheme $X$, any $i,j\in\N$ and any line bundle $\mathcal L$ over $X$, we have
\[
\Phi_{i,j,\mathcal L}=(\Phi_{i,j}+\overline c_1(\mathcal L)\cup).
\]
\end{thm}

\begin{proof}
In outline, the proof will proceed as follows.  We consider the total space of the line bundle ${\mathcal L}$ over $X$.  By pull-back stability of the operation and homotopy invariance, we can relate the operation $\Phi_{i,j,{\mathcal L}}$ with the operation $\Phi_{i,j}$ on the total space of the line bundle ${\mathcal L}$, with a twist coming from the first Chern class of the line bundle via the various identifications.  To establish the result, we track the action of $\Phi_{i,j,{\mathcal L}}$ on suitable explicit representatives of cohomology classes through the identifications just mentioned; for this, we use symmetric complexes and some ideas of Balmer.

As in the proof of Proposition \ref{prop:nontriviality}, we consider the Gersten-Witt complex of $X$ (filtered by powers of the fundamental ideal) $C(X,{\mathbf I}^j(\mathcal L))$:
\[
\xymatrix@C=1.2em{I^j(\mathcal L)(k(X))\ar[r]^-{d^{\mathcal L}_0} & \displaystyle{\bigoplus_{x\in X^{(1)}} I^{j-1}(\mathcal L)_{fl}(k(x))}\ar[r]^-{d^{\mathcal L}_1} &\displaystyle{\bigoplus_{x\in X^{(2)}} I^{j-2}(\mathcal L)_{fl}(k(x)) }\ar[r]^-{d^{\mathcal L}_2} & \ldots.
}
\]
In the case where $\mathcal L=\mathcal O_X$, we will drop ${\mathcal L}$ from the notation. Recall that there is an exact sequence of complexes
\[
0 \longrightarrow C(X,{\mathbf I}^{j+1}(\mathcal L)) \longrightarrow C(X,{\mathbf I}^j(\mathcal L)) \longrightarrow C(X,\bar {\mathbf I}^j) \longrightarrow 0.
\]
If $\alpha\in H^i(X,\bar {\mathbf I}^j)$, then $\Phi_{i,j,\mathcal L}(\alpha)$ is defined as follows. If $\alpha^\prime\in C^i(X,{\mathbf I}^j(\mathcal L))$ is any lift of $\alpha$, then its boundary $d^{\mathcal L}_i(\alpha^\prime)\in C^{i+1}(X,{\mathbf I}^j(\mathcal L))$ is the image of a unique cycle $\beta\in C^{i+1}(X,{\mathbf I}^{j+1}(\mathcal L))$. The reduction of $\beta$ in $C^{i+1}(X,\bar {\mathbf I}^{j+1})$ is precisely $\Phi_{i,j,\mathcal L}(\alpha)$.

Let us observe next that if $p:L\to X$ is the total space of $\mathcal L$, then $p$ induces morphisms of complexes $p^*$ (\cite[Corollaire 9.3.2]{FaselChowWitt}) fitting into the following commutative diagram:
\[
\xymatrix{0\ar[r] & C(X,{\mathbf I}^{j+1}(\mathcal L))\ar[r]\ar[d]^-{p^*} & C(X,{\mathbf I}^j(\mathcal L))\ar[r]\ar[d]^-{p^*} & C(X,\bar {\mathbf I}^j)\ar[r]\ar[d]^-{p^*} & 0 \\
0\ar[r] & C(L,{\mathbf I}^{j+1}(p^*\mathcal L))\ar[r] & C(L,{\mathbf I}^j(p^*\mathcal L))\ar[r] & C(L,\bar {\mathbf I}^j)\ar[r] & 0.}
\]
By homotopy invariance, the vertical morphisms induce isomorphisms on cohomology groups by \cite[Th\'eor\`eme 11.2.9]{FaselChowWitt}.  We use these identifications to replace $X$ by $L$ in what follows.

Now, let us recall how to obtain explicit representatives for elements of $H^i(X,\bar {\mathbf I}^j)$; this involves the formalism of \cite{Balmer02}.  Let $\alpha\in H^i(X,\bar {\mathbf I}^j)$ and let $\alpha^\prime\in C^i(X,{\mathbf I}^j)$ be a lift of $\alpha$. Under the equivalences of \cite[Theorem 6.1, Proposition 7.1]{Balmer02}, $\alpha^\prime$ can be seen as a complex $P_\bullet$ of finitely generated $\mathcal O_X$-locally free modules, together with a symmetric morphism (for the $i$-th shifted duality)
\[
\psi:P_\bullet \longrightarrow T^i\mathrm{Hom}(P_\bullet,\mathcal O_X)
\]
whose cone is supported in codimension $\geq i+1$. By definition, $d_i(\alpha^\prime)$ is the localization at the points of codimension $i+1$ of the symmetric quasi-isomorphism on the cone of $\psi$ (constructed for instance in \cite[Proposition 1.2]{Balmer02}), after d\'evissage (\cite[Theorem 6.1, Proposition 7.1]{Balmer02}).

The first Chern class of $p^*{\mathcal L}$ appears in a natural way using this language.  We want to choose a representative of $\bar{c}_1(p^*{\mathcal L})$ in $H^1(L,\K^{\Mi}_1/2) \cong H^1(L,\bar {\mathbf I})$.  A lift of this element to $H^1(L,{\mathbf I}(p^*\mathcal L))$ can be described as follows.  The zero section
\[
s: \mathcal O_L \longrightarrow p^*\mathcal L
\]
can be seen as a symmetric morphism $\mathcal O_L\to \mathrm{Hom}_{\mathcal O_X}(\mathcal O_L,p^*\mathcal L)$, which is an isomorphism after localization at the generic point of $L$, and whose cone is supported in codimension $1$. It follows that $s$ can be thought of as an element of $C^0(L,{\mathbf W}(p^*\mathcal L))$.  The class of $d_0^{p^*\mathcal L}(s)$ can be viewed as an element of $H^1(L,{\mathbf I}(p^*\mathcal L))$, and its projection in $H^1(L,\bar {\mathbf I})=Pic(L)/2$ is precisely the first Chern class of $p^*\mathcal L$ (\cite[proof of Lemma 3.1]{FaselIJ}).

To lift an element $\alpha\in H^i(X,\bar {\mathbf I}^j)$ to an element $\alpha^\prime$ in $C^i(X,{\mathbf I}^j(\mathcal L))$, we will first find a lift $\alpha^{\prime\prime}\in C^i(X,{\mathbf I}^j)$ and then multiply by $s$ in a sense to be explained more carefully below to obtain our lift $\alpha^\prime$.  A Leibniz-type formula can then be used to compute the boundary of this product and derive the formula in the statement of our theorem.

Using the product structure (say the left one) on derived categories with duality of \cite{Gille03c}, we can obtain an element of $C^i(L,{\mathbf I}^j(p^*\mathcal L))$ lifting $p^*\alpha\in H^i(L,\bar {\mathbf I}^j)$ using the symmetric morphism
\[
p^*\psi\otimes s:p^*P_\bullet \longrightarrow T^i\mathrm{Hom}(p^*P_\bullet,p^*\mathcal L).
\]
The degeneracy locus of $p^*\psi$ in the sense of \cite[Definition 3.2]{Balmer05a} is, by definition, the support of its cone, which has codimension $\geq i+1$ in $L$. The degeneracy locus of $s$ has codimension $1$ in $L$ and intersects the degeneracy locus of $p^*\psi$ transversally.

Now, we are in a position to apply the Leibniz formula of \cite[Theorem 5.2]{Balmer05a} (while the hypotheses of the quoted result are not satisfied in our situation, the proof of \cite[Propostion 4.7]{Fasel07} explains why the formula continues to hold in the case where the intersection of degeneracy loci is transversal).  Since we will momentarily consider the sheaf $\bar {\mathbf I}^j$ whose cohomology groups are $2$-torsion, we can ignore signs, in which case the Leibniz formula gives the equality:
\begin{equation}\label{eqn:main}
d_i^{p^*\mathcal L}(p^*\psi\otimes s)=d_i(p^*\psi)\otimes s + p^*\psi\otimes d_0^{p^*\mathcal L}(s)
\end{equation}
in $C^{i+1}(L,{\mathbf I}^j(p^*\mathcal L))$.

Since $p^*\alpha\in H^i(L,\bar {\mathbf I}^j)$ and $p^*\psi\otimes s$ lifts $p^*\alpha$ in $C^i(L,{\mathbf I}^j(p^*\mathcal L))$ it follows that $d_i^{p^*\mathcal L}(p^*\psi\otimes s)$ actually belongs to $C^{i+1}(L,{\mathbf I}^{j+1}(p^*\mathcal L))$. For the same reason, we have $d_i(p^*\psi)\in C^{i+1}(L,{\mathbf I}^{j+1})$ and then $d_i(p^*\psi)\otimes s\in C^{i+1}(L,{\mathbf I}^{j+1}(p^*\mathcal L))$. Thus $p^*\psi\otimes d_0^{p^*\mathcal L}(s)$ is in $C^{i+1}(L,{\mathbf I}^{j+1}(p^*\mathcal L))$ as well. It follows that all three terms in (\ref{eqn:main}) define classes in $C^{i+1}(L,\bar {\mathbf I}^{j+1})$. The left term yields a class in $H^{i+1}(L,\bar {\mathbf I}^{j+1})$ which is $\Phi_{i,j,p^*\mathcal L}(p^*\alpha)$ by definition. The middle term projects to $\Phi_{i,j}(p^*\alpha)$ and the right-hand term to the class $p^*\alpha\cdot c_1(p^*\mathcal L)$ in $H^{i+1}(L,\bar {\mathbf I}^{j+1})$.
\end{proof}

\section{Differentials, cohomology operations and the Euler class}
Having established the basic properties of the differentials in the Pardon spectral sequence, we now pass to their identification with known operations on motivic cohomology.

\subsection{Differentials in terms of motivic cohomology}
Let us first recall some notation.  Write $\mathcal H^j$ for the Zariski sheaf associated with the presheaf $U\mapsto H^j_{\et}(U,\Z/2)$.  For integers $p,q$, write $H^{p,q}(X,\Z/2)$ for the motivic cohomology groups with $\Z/2$ coefficients as defined by Voevodsky (see, e.g., \cite[Lecture 3]{MVW}); these groups are by construction hypercohomology of certain complexes of Zariski sheaves.
We begin by recalling a result of Totaro \cite[Theorem 1.3]{Totaro}.

\begin{thm}\label{thm:deep}
Suppose $k$ is a field having characteristic unequal to $2$, and $X$ is a smooth $k$-scheme. For any integer $j\geq 0$, there is a long exact sequence of the form:
\[
\xymatrix@C=1.1em{\ldots\ar[r] & H^{i+j,j-1}(X,\Z/2)\ar[r] & H^{i+j,j}(X,\Z/2)\ar[r] & H^i(X,\mathcal H^j)\ar[r] & H^{i+j+1,j-1}(X,\Z/2)\ar[r] & \ldots};
\]
this exact sequence is functorial in $X$.
\end{thm}

\begin{proof}[Comments on the proof.]
This result requires Voevodsky's affirmation of Milnor's conjecture on the mod $2$ norm residue homomorphism \cite{VMilnor} as well as the Beilinson-Lichtenbaum conjecture, which is equivalent to the Milnor conjecture by results of Suslin-Voevodsky and Geisser-Levine.  The functoriality assertion of the statement is evident from inspection of the proof (it appears by taking hypercohomology of a distinguished triangle). \end{proof}

We will use the above exact sequence in the guise established in the following result.

\begin{cor}\label{cor:deep}
For any $i\in\N$ and any smooth scheme $X$ over a perfect field $k$ with $\mathrm{char}(k)\neq 2$, the above sequence induces an isomorphism
\[
H^{2i+1,i+1}(X,\Z/2)\simeq H^{i}(X,\bar {\mathbf I}^{i+1})
\]
that is functorial in $X$.
\end{cor}

\begin{proof}
The exact sequence of Theorem \ref{thm:deep} reads as follows for $j=i+1$
\[
\xymatrix@C=1.1em{\ldots\ar[r] & H^{2i+1,i}(X,\Z/2)\ar[r] & H^{2i+1,i+1}(X,\Z/2)\ar[r] & H^i(X,{\mathcal H}^{i+1})\ar[r] & H^{2i+2,i}(X,\Z/2)\ar[r] & \ldots}
\]
Since $k$ is perfect, we have $H^{p,i}(X,\Z/2)=0$ for any $p\geq 2i+1$ by \cite[Theorem 19.3]{MVW} and from this we can conclude that the middle arrow is an isomorphism.  Since the exact sequence is functorial in $X$, it follows immediately that the isomorphism just mentioned has the same property.  Now, the affirmation of the Milnor conjecture on the mod $2$ norm residue homomorphism also implies that $\K^{\Mi}_{i+1}/2$ can be identified as a sheaf with ${\mathcal H}^{i+1}$, while the affirmation of the Milnor conjecture on quadratic forms yields an identification of sheaves $\K^{\Mi}_{i+1}/2 \cong \bar {\mathbf I}^{i+1}$.  Combining these isomorphisms yields an isomorphism ${\mathcal H}^{i+1} \cong \bar {\mathbf I}^{i+1}$ and therefore an identification of cohomology with coefficients in these sheaves functorial in the input scheme.
\end{proof}

Voevodsky defined in \cite[p. 33]{VRed} motivic Steenrod operations $Sq^{2i}: H^{p-2i,q-i}(X,\Z/2) \to H^{p,q}(X,\Z/2)$.  The resulting operations are bi-stable in the sense that they are compatible with $\pone$-suspension in the same sense as described in the previous section.  Via the isomorphism of Corollary \ref{cor:deep}, we can view $Sq^2$ as an operation
\[
Sq^2: H^{i-1}(X,\bar {\mathbf I}^{i}) \longrightarrow H^{i}(X,\bar {\mathbf I}^{i+1}),
\]
which is again bi-stable in the sense that it is compatible with $\pone$-suspension.  The algebra of bistable cohomology operations in motivic cohomology with $\Z/2$-coefficients was determined by Voevodsky in characteristic $0$, \cite{VMEMSpaces} and extended to fields having characteristic unequal to $2$ in \cite{HKO}.  Using these results, we may now identify the operation $\Phi_{i-1,i}$ described in Definition \ref{defn:phiijL} in more explicit terms.

\begin{cor}\label{cor:sq2}
\label{cor:connectinghomomissq2}
We have an identification $\Phi_{i-1,i} = Sq^2$.
\end{cor}

\begin{proof}
The operation $\Phi_{i-1,i}$ is bistable by Proposition \ref{prop:bistable}, commutes with pullbacks by construction, and changes bidegree by $(2,1)$ so it is pulled back from a universal class on a motivic Eilenberg-Mac Lane space.  On the other hand, the group of bi-stable operations of bidegree $(2,1)$ is isomorphic to $\Z/2$ generated by $Sq^2$: if $k$ has characteristic zero, this follows from \cite[Theorem 3.49]{VMEMSpaces}, while if $k$ has characteristic unequal to $2$, this follows from \cite[Theorem 1.1]{HKO}.  Since the operation $\Phi_{i-1,i}$ is non-trivial by Proposition \ref{prop:nontriviality}, it follows that it must be equal to $Sq^2$.
\end{proof}

The next result is an immediate consequence of Corollary \ref{cor:sq2} and Theorem \ref{thm:twistplus}.

\begin{thm}
\label{thm:main}
Suppose $k$ is a field having characteristic unequal to $2$, and $X$ is a smooth $k$-scheme.  For any integer $i > 0$, and any rank $r$ vector bundle $\xi: \mathcal{E} \to X$, the operation $(Sq^2 + \overline c_1(\xi) \cup)$ coincides with $\Phi_{i-1,i,\det \xi}$.
\end{thm}


\subsection{The Euler class and secondary classes}
\label{s:eulerclasses}
The Euler class $e(\mathcal E)$ of a rank $d$ vector bundle $\xi: \mathcal{E} \to X$ is the only obstruction to splitting off a free rank $1$ summand, and it lives in $H^d(X,\K^{\MW}_d(\mathcal L))$ where $\mathcal L=\det \mathcal{E}$. Now, the Euler class is mapped to the top Chern class $c_d(\mathcal E)$ in $CH^d(X)$ under the homomorphism $H^d(X,\K^{\MW}_d(\mathcal L))\to H^d(X,\K^{\Mi}_d)=CH^d(X)$ induced by the morphism of sheaves $\K^{\MW}_d(\mathcal L)\to \K^{\Mi}_d$ and it follows that the vanishing of $e(\mathcal E)$ guarantees vanishing of $c_d(\mathcal E)$ in $CH^d(X)$ (see \cite[Proposition 6.3.1]{AsokFaselA3minus0} for this statement).

The vanishing of the top Chern class does not, in general, imply vanishing of the Euler class, as shown by the example of the tangent bundle to the real algebraic sphere of dimension $2$. For vector bundles with vanishing top Chern class, we can use Theorem \ref{thm:obstructions} to decide whether its Euler class vanishes, provided we work over a field $k$ of finite $2$-cohomological dimension. In the next theorem, we denote by $\Psi^n(\mathcal E)$ the obstruction classes $\Psi^n(e(\mathcal E))$ of Theorem \ref{thm:obstructions} associated to the Euler class $e(\mathcal E)$.

\begin{thm}
\label{thm:main2}
Suppose $k$ is a field having finite $2$-cohomological dimension, $X$ is a smooth $k$-scheme of dimension $d$ and $\xi: \mathcal{E} \to X$ is a rank $d$ vector bundle on $X$ with $c_d(\mathcal E) = 0$. The vector bundle $\mathcal E$ splits off a trivial rank $1$ summand if only if, in addition, $\Psi^n(\mathcal E) = 0$ for $n\geq 1$.
\end{thm}

As mentioned in the introduction, the advantage of the computation of these higher obstruction classes over the computation of the Euler class is that the cohomology groups involved are with coefficients in cycle modules in the sense of Rost, which are a priori more manageable than cohomology with coefficients in more exotic sheaves such as Milnor-Witt $K$-theory. Moreover, Corollary \ref{cor:deep} shows that the differentials, at least in some range, can be identified with Steenrod operations, which are arguably more calculable. The obvious weakness of this approach is the appearance of the groups $\overline G_i$ defined in Section \ref{sec:Pardon}, though see Remark \ref{rem:2torsion} for a counterpoint.  Continuing with the assumption that our base field $k$ has finite cohomological dimension one can show that establishing the vanishing of finitely many obstructions (depending on the cohomological dimension) are sufficient to guarantee vanishing of all obstructions.  The next result completes the verification of Corollary \ref{corintro:finitelymanyobstructions} from the introduction.

\begin{cor}\label{cor:main2}
Assume $k$ is a field of $2$-cohomological dimension $s$, $X$ is a smooth $k$-scheme of dimension $d$ and $\xi:\mathcal E\to X$ is a rank $d$-vector bundle over $X$ with $c_d(\mathcal E)=0$.  The vector bundle $\mathcal E$ splits off a trivial rank $1$ summand if and only if $\Psi^n(\mathcal E) = 0$ for $n\leq s-1$.
\end{cor}

\begin{proof}
In view of the definition of the higher obstructions $\Psi^n(\mathcal E)$, it suffices to show that $H^d(X,\bar {\mathbf I}^j)$ vanishes for $j\geq d+r$. This is \cite[Proposition 5.2]{AsokFaselThreefolds}, together with the identification of Nisnevich and Zariski cohomology with coefficients in $\bar {\mathbf I}^j$ explained in \cite[\S 2]{AsokFaselThreefolds}.
\end{proof}

Finally, combining all of the results established so far, we can complete the verification of Theorem \ref{thmintro:secondaryclass}.

\begin{proof}[Completion of proof of Theorem \ref{thmintro:secondaryclass}]
To identify the secondary obstruction $\Psi^1$ as the composition in the statement, we begin by observing that the group $E(\mathcal L,\MW)_\infty^{d,d+1}$ is the cokernel of the composite map
\[
\xymatrix@C=4em{H^{d-1}(X,\K_d^{\Mi})\ar[r] & H^{d-1}(X,\K_d^{\Mi}/2)\ar[r]^-{Sq^2+\overline c_1(\mathcal L)} & H^d(X,\K_{d+1}^M/2)}
\]
in view of Lemma \ref{lem:KMWss} and Theorem \ref{thm:main}.

If $k$ has cohomological dimension $\leq 1$, it follows immediately from Corollary \ref{cor:main2} that the top Chern class is the only obstruction to splitting a free rank $1$ summand.

If $k$ has cohomological dimension $\leq 2$, Theorem \ref{thm:main2} says in this context that the Euler class of $E$ (take $\L = \mathrm{det}(E)$) is trivial if and only if the top Chern class and the first obstruction class in $E(\mathcal L,\MW)_\infty^{d,d+1}$ vanish.
\end{proof}

\begin{footnotesize}
\bibliographystyle{alpha}
\bibliography{secondarycharacteristicclasses}
\end{footnotesize}
\end{document}